\documentclass[11pt,letterpaper]{amsart}
\usepackage{amsthm}

\usepackage{mathrsfs}
\usepackage{mathtools}
\usepackage{slashed}
\usepackage{tikz-cd}

\usepackage{amssymb}
\usepackage{amsfonts}
\usepackage{amsmath}
\usepackage{graphicx}
\usepackage{xcolor}
\usepackage{amsmath}
\usepackage{mathrsfs}
\usepackage{stmaryrd}
\usepackage{epsfig,color}
\usepackage{blindtext}
\usepackage{enumerate}
\usepackage[colorlinks=true,linkcolor=red,citecolor=blue]{hyperref}
\usepackage[noabbrev,capitalize]{cleveref}

\usepackage{url}
\usepackage{nicefrac,mathtools}
\DeclareGraphicsExtensions{.pdf,.jpeg,.png}
\usepackage{epstopdf}
\usepackage{cancel} 
\usepackage[normalem]{ulem} 
\usepackage{verbatim} 
\usepackage{enumitem} 
\usepackage{color}
\usepackage[msc-links, lite, alphabetic]{amsrefs}
\usepackage{geometry}
\DeclareMathOperator{\divsymb}{div}

\DeclareMathOperator{\tr}{\rm tr}

\DeclareMathOperator{\vol}{Vol}

\DeclareMathOperator{\Ric}{Ric}

\DeclareMathOperator{\Hess}{\mathrm{Hess}}

\renewcommand{\subset}{\subseteq}

\newcommand{\del}{\partial}

\newcommand{\lp}{\langle}
\newcommand{\rp}{\rangle}

\newcommand{\mL}{\mathcal{L}}

\newcommand{\kQ}{\mathfrak{Q}}

\newcommand{\kq}{\mathfrak{q}}

\newcommand{\mb}{\mathbb}

\newcommand{\n}{\mathbf n}

\newcommand{\R}{\mathbb{R}}
\renewcommand{\subset}{\subseteq}

\newcommand{\dist}{\mathsf{d}}

\newcommand{\di}{\mathop{}\!\mathrm{d}}

\DeclareMathOperator{\RCD}{RCD}

\def\XXint#1#2#3{{\setbox0=\hbox{$#1{#2#3}{\int}$ }
\vcenter{\hbox{$#2#3$ }}\kern-.6\wd0}}

\def\sideremark#1{\ifvmode\leavevmode\fi\vadjust{\vbox to0pt{\vss
 \hbox to 0pt{\hskip\hsize\hskip1em
 \vbox{\hsize3cm\tiny\raggedright\pretolerance10000
 \noindent #1\hfill}\hss}\vbox to8pt{\vfil}\vss}}}

\newtheorem{theorem}{Theorem}[section]

\newtheorem{proposition}[theorem]{Proposition}
\newtheorem{lemma}[theorem]{Lemma}

\theoremstyle{definition}
\newtheorem{definition}[theorem]{Definition}

\newtheorem{conjecture}[theorem]{Conjecture}
\newtheorem{remark}[theorem]{Remark}

\numberwithin{equation}{section}

\allowdisplaybreaks[4]

\title{Nonnegative Ricci Curvature and Uniformly Convex Boundary Forces Compactness}
\author[Z. Yan]{Zetian Yan}
\address[Z. Yan]{Department of Mathematics \\ UC Santa Barbara \\ Santa Barbara \\ CA 93106 \\ USA}
\email{ztyan@ucsb.edu}

\author[X. Zhu]{Xingyu Zhu}
\address[X. Zhu]{Department of Mathematics \\ Michigan State University \\ East Lansing \\ MI 48824 \\ USA}
\email{zhuxing3@msu.edu}


\begin{document}

\begin{abstract}
We confirm a compactness conjecture of M. Li. If a complete Riemannian manifold has nonnegative Ricci curvature and uniformly convex boundary in the sense that the second fundamental form satisfies $h\ge1$. Then we prove it is compact, and consequently has finite fundamental group. The proof uses monotone quantities constructed via positive proper harmonic functions with Neumann condition.
\end{abstract}

\maketitle

\section{Introduction}

\subsection{Backgrounds and statements.} A central theme in Riemannian geometry is to understand how lower curvature bounds constrain the global geometry and topology of a complete manifold. We are concerned with this philosophy for manifolds with boundary. A natural condition is to combine a lower Ricci curvature bound in the interior with a convexity or mean convexity assumption on the boundary. 

A stunning result about the combined effect of interior and boundary curvature lower bound is obtained by Hamilton \cite{HamiltonCompactness}, which states that a convex hypersurface in $\R^n$ with uniformly pinched second fundamental must be compact.


The purpose of this paper is to prove the conjecture of M. Li that substantially generalizes the result of Hamilton. This conjecture is also a version of Bonnet--Myers theorem for manifolds with boundary, where the strict positivity is imposed on the boundary. It is stated as follows. 

\begin{conjecture}[\cite{LiMartin2014}*{Conjecture 1.3}]\label{conj:Li}
Let $(M^n,g)$ be a complete Riemannian manifold with smooth nonempty  boundary
$\partial M$. Suppose $M$ has nonnegative Ricci curvature and $\partial M$ is uniformly convex i.e., the second fundamental form $h$ satisfies  $h\ge k>0$ for some constant $k$. Then $M$ is compact and $\pi_1(M)$ is finite.
\end{conjecture}

If one strengthens the nonnegative Ricci curvature assumption to nonnegative sectional curvature, the corresponding conjecture is proven by \cite{RadiusAlexandrov}. This includes the two-dimensional case of Conjecture~\ref{conj:Li}. An alternative variational proof is given by Y. Wu~\cite{Wu2025}; see also \cite{BuragoZalgaller1977}. In fact \cite{RadiusAlexandrov} and \cite{Wu2025} also obtained the sharp area upper bound of the boundary, and obtained rigidity of the boundary in the equality case. The sharp boundary area upper bound is closely related to a conjecture of Wang \cite{Wang2021}*{Conjecture 2}, which asks for a compact manifold $M$ with boundary $\del M$ satisfies  $\Ric\ge 0$, $h\ge 1$ if $|\partial M|\le |\mb S^{n-1}|$ holds with equality iff $M$ is isometric to the Euclidean disk $\mb B^n$.
 There are several work supporting the strong rigidity of the set of assumptions $\Ric\ge 0$, $h\ge 1$ for compact manifolds. For example, in \cite{Fraser-Li2014}, Fraser--Li obtained compactness theorems for embedded free-boundary minimal surfaces in compact $3$-manifolds with nonnegative Ricci curvature and strictly convex boundary. See also \cites{HangWang2009, Wang2021}.


Throughout the paper, we denote by $\nu$ the outward unit normal to $\partial M$ and use the convention
\[
        h(X,Y)=\langle \nabla_X\nu,Y\rangle,
        \qquad X,Y\in T\partial M,
\]
for the second fundamental form of the boundary. With this convention, the standard Euclidean ball has $h>0$. After rescaling the metric, we assume $h\ge 1$.

Our main result confirms Conjecture~\ref{conj:Li}.

\begin{theorem}\label{thm:main}
Let $(M^n,g)$ be a complete Riemannian $n$-manifold with nonempty smooth
boundary $\partial M$. Suppose that
\[
        \Ric_g\geq 0,
        \qquad
        h\geq 1.
\]
Then $M$ is compact. Consequently, $\pi_1(M)$ is finite.
\end{theorem}

The compactness conclusion is the main part of the theorem. The finiteness of $\pi_1(M)$ follows by passing to the universal cover. Indeed, the universal cover of $M$ satisfies the same assumptions in the main theorem, hence is also compact. Then it follows that $\pi_1(M)$ is finite.

\subsection{Strategy and Proof outline.} The motivation for our proof comes from level set methods in geometric
analysis. 
In \cites{Colding2012,ColdingMinicozzi2014,ColdingMinicozzi2013Proceedings,ColdingMinicozzi2014CVPDE,ColdingMinicozzi2014Inventiones,Colding-Minicozzi2025}, Colding--Minicozzi developed monotonicity formulas along level sets of Green's functions on manifolds with nonnegative Ricci curvature, with applications to the asymptotic cone of Ricci flat manifolds. Munteanu--Wang also used Green function level
sets in their study of $3$-manifolds with scalar curvature lower
bounds \cites{MunteanuWang2023, MunteanuWang2025, MunteanuWang2026}. In \cite{BrayKazarasDemetreKhuri2022}, Bray--Kazaras--Khuri--Stern used harmonic functions to study scalar curvature and the ADM mass;
see also the related works
\cites{AgostinianiMazzieriOronzio2024,AgostinianiMantegazzaMazzieriOronzio2023} on this topic. 
These works illustrate that a well-chosen ``potential'' can serve as a substitute for the distance function, and that its
level sets can encode information about the geometry at infinity. 

The present paper is also inspired by our previous work \cites{YanZhu2026AsymptoticLimitsII,YanZhu2026Topology}, which is developed based on the ideas in \cites{Zhu2025,YanZhuuniqueness}. We observed through the previous work that one needs to treat separately the parabolic, nonparabolic dichotomy.

\begin{definition}
    Let $(M^n,g)$ be a complete noncompact Riemannian manifold, $M$ is parabolic if it admits no positive Green's function. Otherwise, $M$ is nonparabolic. 
\end{definition}
In the parabolic case, we constructed monotone quantities parallel to those constructed by Colding--Minicozzi in the nonparabolic case, and used them to study the topology of $3$-manifolds with nonnegative scalar curvature. In \cite{YanZhu2026Topology}, we studied complete $3$-manifolds with nonnegative scalar curvature through exhaustions by level sets of positive harmonic functions or Green distance functions. One of the main observations there, as well as in \cite{Colding-Minicozzi2025}, is that the growth rate of a monotone level set integral encodes geometric information about an end, which in turn controls the topology of the end in dimension $3$ via the Gauss--Bonnet theorem.

In the present paper, we develop this point of view to manifolds with boundary. The potentials used here are required to be compatible with the boundary geometry: we impose the Neumann boundary condition. More precisely, in the parabolic case, we use a proper positive harmonic function satisfying the Neumann boundary condition outside a compact set, while in the nonparabolic case, we use the positive Neumann Green function and its associated Green distance function.

\begin{remark}
We emphasize that the noncompactness is essential in our argument. Indeed, the construction of the Evans potential, the decay estimate for the Neumann Green's function and our contradicting argument all rely crucially on the noncompactness, which is imposed as the hypothesis for contradiction. This means that our approach does not provide insights into Wang's conjecture on the (sharp) upper bound of the boundary area.
\end{remark}

We now describe the main ideas of the proof. We assume for the sake of contradiction that $M$ is complete and noncompact. The argument to derive a contradiction is divided according to the dichotomy between the parabolic and nonparabolic
cases.

We first explain the parabolic case. The starting point is the construction of a proper positive harmonic function satisfying the Neumann boundary condition
outside a compact set. More precisely, by
Proposition~\ref{mixed-harmonic-function}, if $M$ is parabolic, then for every nonpolar compact set $K\subset M$ there exists a positive harmonic function $u$ on $M\setminus K$ such that
\[
        \Delta u=0
        \quad\text{in } M\setminus K,
        \qquad
        \partial_\nu u=0
        \quad\text{on } \partial M\setminus K,
\]
with $u=0$ on $K$ and $u(x)\to\infty$ as $x\to\infty$. Since all positive level sets of $u$ stay away from $K$, one can use them to probe the geometry of the end.

For positive $t$, we consider the level-set functional
\[
        A(t)=\int_{\{u=t\}}|\nabla u|^3\,\di \sigma.
\]
The Neumann condition implies that $\nabla u$ is tangent to $\partial M$ along the boundary. Hence the level sets of $u$ meet $\partial M$ orthogonally, and the boundary contribution in the variation formula can be written in terms of the second fundamental form of $\partial M$. The key monotonicity identity is Theorem~\ref{thm:free-boundary-monotonicity-parabolic}: for positive regular values $t<s$,
\[
A'(s)-A'(t)
=
2\int_{\{t\leq u\leq s\}}
\left(
|\nabla^2u|^2+\Ric(\nabla u,\nabla u)
\right)\,\di V  
+
2\int_{\partial M\cap\{t\leq u\leq s\}}
h(\nabla u,\nabla u)\,\di \sigma .
\]
Thus the assumptions $\Ric\geq0$ and $h\geq0$ imply that $A'(t)$ is monotone nondecreasing. This is the first essential use of the convexity of the boundary.

The stronger assumption $h\geq1$ is used in two further ways. First, the
boundary term in the monotonicity formula quantitatively controls the boundary energy
\[
        \int_{\partial M\cap\{u\geq T\}}|\nabla u|^2\,\di \sigma.
\]
Second, together with $\Ric\geq0$, it yields the $L^2$ trace Poincar\'e inequality proved in Theorem~\ref{Trace Poincare inequality}:
\[
        \int_M f^2\,\di V
        \leq
        \frac{2}{n}\int_{\partial M}f^2\,\di \sigma
        +
        \frac{n+1}{2n^2}\int_M|\nabla f|^2\,\di V.
\]
This inequality converts the finite energy estimates obtained from the monotonicity formula into a contradiction.

The final contradiction in the parabolic case is obtained by considering the sign of $A'(t)$ at infinity. If $A'(t)>0$ at some sufficiently large level, then we derive a differential inequality that forces the superquadratic growth
\[
        A(t)\geq c\,t^{\frac{2n-2}{n-2}}
\]
for all sufficiently large $t$, by Lemma \ref{ODE-lemma}. On the other hand, the Cheng--Yau gradient estimate gives the quadratic upper bound
\[
        A(t)\leq C t^2,
\]
which is impossible when $n\geq3$.

It remains to rule out the alternative that $A'(t)\leq0$ for all sufficiently large $t$. Since $A'(t)$ is monotone nondecreasing and $A(t)\geq0$, it follows $\lim_{t\to \infty} A'(t)=0$. Letting the outer level tend to infinity in Theorem~\ref{thm:free-boundary-monotonicity-parabolic} gives finite interior Hessian energy and finite boundary energy on the end:
\[
        \int_{\{u\geq T\}}|\nabla^2u|^2\,\di V
        +
        \int_{\partial M\cap\{u\geq T\}}|\nabla u|^2\,\di \sigma
        <\infty .
\]
Applying the trace Poincar\'e inequality Theorem \ref{app:trace-poincare}  to a cutoff of $u$ then implies
\[
        \int_{\{u\geq T\}}|\nabla u|^2\,\di V<\infty.
\]
This contradicts the coarea formula and the flux identity:
\[
        \int_{\{u\geq T\}}|\nabla u|^2\,\di V
        =
        \int_T^\infty
        \left(
        \int_{\{u=t\}}|\nabla u|\,\di \sigma
        \right)\,\di t = \infty.
\]
Therefore the parabolic noncompact case is impossible.

The nonparabolic case follows the same strategy, with two essential changes. First, the Evans potential is replaced by the minimal positive Neumann Green's function. More precisely, by Proposition~\ref{prop:existence-properness-Neumann-Green}, after fixing a pole $p\in \mathring{M}$, there exists a minimal positive Neumann Green function $G=G_N(p,\cdot)$, and the associated Green distance
\[
        b=G^{\frac{1}{2-n}}
\]
is proper under the assumptions $\Ric\geq0$ and $h\geq0$. 

Second, in nonparabolic setting, the model is the Euclidean space $\R^n$, where $b$ is just the distance function to the pole. The natural symmetry for level sets of $b$ is scaling. Inspired by this, we define
\[
        A(t)=t^{1-n}\int_{\{b=t\}}|\nabla b|^3\,\di \sigma.
\]
The natural monotone quantity is not $A'(t)$ itself, but the weighted derivative
\[
        F(t):=t^{3-n}A'(t).
\]
The corresponding monotonicity formula is Theorem~\ref{thm:weighted-second-variation-nonparabolic}. It contains a nonnegative bulk term involving the trace-free tensor
\[
        B=\nabla^2 b^2-2|\nabla b|^2g
\]
and a boundary term controlled by the second fundamental form of $\partial M$. Thus $\Ric\geq0$ and $h\geq0$ imply that $F(t)$ is monotone nondecreasing, while the stronger assumption $h\geq1$ gives the quantitative boundary control needed later.

The contradiction argument then parallels the parabolic case. If $F(t)>0$ at some sufficiently large level, then the weak weighted growth lemma, Lemma~\ref{lem:weighted-growth-nonparabolic}, forces
\[
        A(t)\geq c\,t^{2n-4}
\]
for all large $t$. This contradicts the Cheng--Yau gradient estimate for the Green distance and the Green flux identity (\ref{conservative law}), which together give a uniform upper bound for $A(t)$. If instead $F(t)\leq0$ at infinity, then the monotonicity formula yields finite scale-invariant energy estimates for $B$ and for the boundary term. Combining these estimates with the trace Poincar\'e inequality contradicts the divergence (\ref{linear-divergence}) forced by the Green flux identity (\ref{conservative law}). Therefore the nonparabolic noncompact case is also impossible.

The paper is organized as follows.  In Section~\ref{sec:parabolic}, we treat the parabolic case.  We first construct the Evans potential with Neumann boundary condition, then derive the level-set monotonicity formula and complete the contradiction argument. In Section~\ref{sec:nonparabolic}, we treat the nonparabolic case using the minimal Neumann Green function and the associated Green distance.  Finally, in Appendix~\ref{app:trace-poincare}, we prove the trace Poincar\'e inequality used in both cases.

\subsection*{Acknowledgments}

The authors are grateful to Xianzhe Dai, Han Hong, Martin Li, Jian Wang, Xiaodong Wang and Guofang Wei for their interests in this problem and for helpful comments. The authors also thank the organizing committee of the conference \emph{When Analysis Meets Geometry} and UC Santa Cruz department of Mathematics, where a large part of this work was completed, for their hospitality and support. Z.Y. and X.Z. are supported by AMS--Simons Travel Grants.

\section{Parabolic case}\label{sec:parabolic}
In this section, we verify Li's conjecture in the parabolic case.

\begin{proposition}\label{prop:Parabolic confirmation of Li's conjecture}
    Let $(M^n,g)$ be a complete Riemannian $n$-manifold with nonempty boundary $\partial M$ and satisfies
    \begin{equation}
        \Ric_g\ge 0, \quad h\ge 1.
    \end{equation}
    Then $(M,g)$ cannot be both noncompact and parabolic.
\end{proposition}

The rest of this section is devoted to the proof of Proposition~\ref{prop:Parabolic confirmation of Li's conjecture}. We will prove by contradiction, so in the following we always assume that $(M,g)$ is noncompact and parabolic.
The proof uses a monotone quantity constructed from a proper positive harmonic function satisfying the Neumann boundary condition outside a sufficiently large compact set. We first establish the existence of such a harmonic function.

\subsection{Existence of proper positive Neumann harmonic function}
The function we need to construct is known in the literature as an Evans potential. We shall use an existence theorem of Hansen--Netuka \cite{Hansen-Netuka2013}, which applies in a very general potential-theoretic setting.

\begin{proposition}\label{mixed-harmonic-function}
Let $(M^n,g)$ be a complete noncompact Riemannian manifold with boundary $\partial M$. Suppose that $M$ is parabolic. Then, for every nonpolar compact set $K\subset M$, there exists a positive harmonic function $u$ on $M\setminus K$ such that $ \partial_\nu u=0$ on $\partial M\setminus K$, and such that $u$ extends continuously to $K$ by setting $u=0$ on $K$. Moreover,
\[
    u(x)\to \infty \quad \text{as } x\to \infty.
\]
In particular, $u$ is proper.
\end{proposition}
\begin{remark}
    Removing a compact set is necessary for the existence of a unbounded harmonic function in the parabolic case. For example, take the Euclidean half plane $\R^2_+$. It follows from the Schwarz reflection principle that a Neumann harmonic function on $\R^2_+$ extends to a harmonic function on $\R^2$ and hence a constant by Liouville theorem.
\end{remark}
Recall that a compact set $K$ is called nonpolar if it has positive capacity, $\mathrm{Cap}(K)>0$, where
\[
\mathrm{Cap}(K)
=
\inf\left\{
\|f\|_{W^{1,2}(M)}^2:
f\in W^{1,2}(M),\ 
f\ge 1 \text{ in an open neighborhood of } K,\ 
f\ge 0 \text{ a.e.}
\right\}.
\]
 Although Hansen--Netuka do not explicitly prescribe the value of the Evans potential on $K$, it follows directly from the construction in the proof of \cite{Hansen-Netuka2013}*{Theorem 3.1} that the resulting function may be extended by zero on $K$.

It suffices to verify that a Riemannian manifold with boundary, equipped with the sheaf of Neumann harmonic functions, satisfies the hypotheses of \cite{Hansen-Netuka2013}*{Theorem~1.1}. In fact, Hansen--Netuka stated the existence of Evans potential for parabolic manifolds in \cite{Hansen-Netuka2013}*{Corollary 1.2}. However, they did not specify if such manifolds have boundary or what the boundary condition is.

\begin{proof}
We define the sheaf of Neumann harmonic functions $\mathcal H_N$ as follows. For each open set $U\subset M$, let
\[
  \mathcal H_N(U)
  =
  \left\{
  u\in C^2(U):
  \Delta u=0 \text{ in $U$}, \quad \partial_\nu u=0 \text{ on $U\cap\del M$}
  \right\}.
\]
It is straightforward to check that $\mathcal H_N$ defines a sheaf and that it contains all constant functions. By \cite{Hansen-Netuka2013}*{Theorem~1.1}, it remains to verify that $(M,\mathcal H_N)$ is a Brelot harmonic space. Following \cite{Hansen-Netuka2013}*{Section~2}, this amounts to checking two properties: first, that the open sets which are regular for the Dirichlet problem form a topological basis for $M$; and second, that the Brelot convergence property holds.

Recall that, in the present setting, an open set $U\subset M$ is said to be regular for the Dirichlet problem if, for every continuous function
$f\in C^0(\partial_M U)$, where
\[
    \partial_M U:=\overline U\setminus U
\]
denotes the boundary of $U$ relative to the topological space $M$, the Dirichlet problem
\begin{equation}\label{eq:Local-Dirichlet-Problem}
    \Delta u=0 \quad \text{in } U,\qquad
    u=f \quad \text{on } \partial_M U
\end{equation}
admits a solution $u\in \mathcal H_N(U)$ which extends continuously to $\overline U$ and satisfies $u=f$ on $\partial_M U$.

It suffices to verify this property for a basis of sufficiently small neighborhoods of each point $p\in M$. If $p\in \mathring{M}$, we choose $r>0$ sufficiently small so that $B_g(p,r)\cap \partial M=\emptyset$.
Then, for $U=B_g(p,r)$, problem \eqref{eq:Local-Dirichlet-Problem} is the ordinary local Dirichlet problem for the Laplace equation on a smooth domain. Hence it is solvable by the standard elliptic theory.

It remains to consider the case $p\in \partial M$. Choose $r>0$ sufficiently small and set $U=B_g(p,r)\cap M$. We decompose the boundary of $U$ into the Dirichlet part and the Neumann part by
\[
    \Gamma_D:=\partial_M U=\overline U\setminus U,
    \qquad
    \Gamma_N:=U\cap \partial M.
\]
The local Dirichlet problem for the sheaf $\mathcal H_N$ is therefore equivalent to the mixed boundary value problem
\[
\begin{cases}
    \Delta u=0 & \text{in } U\cap M^\circ,\\
    u=f & \text{on } \Gamma_D,\\
    \partial_\nu u=0 & \text{on } \Gamma_N.
\end{cases}
\]
 Since $r$ can be chosen sufficiently small, $U$ is a smooth half-ball in $\R^n$ up to a change of coordinates, and the mixed Dirichlet--Neumann problem above is solvable by the standard local elliptic theory. Moreover, the solution is continuous up to $\Gamma_D$ and assumes the prescribed boundary values there. Hence such neighborhoods are regular for the Dirichlet problem. Consequently, regular open sets form a topological basis of $M$.

Finally, we verify the Brelot convergence property. This is exactly \cite{GrigorHeatkernel}*{Corollary 13.13}, which is called the Harnack principle there. Indeed, the local harnack inequality \cite{GrigorHeatkernel}*{Theorem 13.10} ensures locally uniform $C^\infty$ convergence for increasing sequence of harmonic functions that are uniformly finite at a point. 


Consequently, $(M,\mathcal H_N)$ is a Brelot harmonic space.
\end{proof}

\subsection{Monotone functional}

We now turn to the monotonicity formula associated with the Evans potential constructed above. 

Let $(M^n,g)$ be a complete Riemannian manifold with nonempty boundary
$\partial M$, and let $K\subset M$ be a nonpolar compact set. Suppose that $u\in C^0(M)\cap C^\infty(M\setminus K)$
is a proper positive harmonic function on $M\setminus K$ such that
\[
    \Delta u=0 \quad \text{in } M\setminus K,
    \qquad
    \partial_\nu u=0 \quad \text{on } \partial M\setminus K,
\]
and such that $u=0$ on $K$. For a positive regular value $t$ of $u$, set
\[
    \Sigma_t=\{u=t\},
    \qquad
    \mathbf n:=\frac{\nabla u}{|\nabla u|},
\]
and define
\[
    A(t):=\int_{\Sigma_t}|\nabla u|^3\,\di \sigma .
\]

Since the harmonic function $u$ is constructed only on the exterior domain $M\setminus K$ and is extended by zero on $K$, all computations below are performed on positive level sets of $u$, which stay away from $K$. 

\begin{theorem}\label{thm:free-boundary-monotonicity-parabolic}
For any two positive regular values $0<t<s$, it holds 
\begin{align}\label{eq:derivative-of-A'-parabolic}
A'(s)-A'(t)
=
2\int_{\{t\le u\le s\}}
\left(
|\Hess u|^2+\Ric(\nabla u,\nabla u)
\right)\,\di V 
+
2\int_{\partial M\cap \{t\le u\le s\}}
h(\nabla u,\nabla u)\,\di \sigma.
\end{align}
In particular, if $\Ric_g\ge 0, h\ge 0$,
then $A'(t)$ is monotone nondecreasing on the set of positive regular values. If moreover $ \lim_{t\to\infty}A'(t)=0$,
then $A'(t)\le 0$ for all positive regular values $t$. Consequently, $A(t)$ is monotone nonincreasing on positive regular values. In particular, $A(t)$ is positive and bounded above.
\end{theorem}

\begin{proof}
Since we only consider positive regular values $0<t<s$, the slab
\[
    \Omega_{t,s}:=\{t\le u\le s\}
\]
is contained in $M\setminus K$. Hence $u$ is smooth and harmonic on a neighborhood of
$\Omega_{t,s}$, and no boundary contribution from $\partial K$ appears in the following
integration by parts.

The Neumann condition $\partial_\nu u=0$ on $\partial M\setminus K$ implies
\begin{align}\label{tangential-on-the-boundary}
    \langle \nabla u,\nu\rangle=0
    \quad \text{on } \partial M\setminus K .
\end{align}
Thus, at a boundary point of a regular level set $\Sigma_t$, the level set normal $\n$ is tangent to $\partial M$.

We first compute the first derivative of $A(t)$. For a smooth function $f$, the standard
first variation formula for regular level sets gives
\[
\frac{\di}{\di t}\int_{\Sigma_t} f\,\di \sigma
=
\int_{\Sigma_t}
\left(
\frac{\langle \nabla f,\nabla u\rangle}{|\nabla u|^2}
+
\frac{fH_{\Sigma_t}}{|\nabla u|}
\right)\,\di \sigma ,
\]
where $H_{\Sigma_t}$ denotes the mean curvature of $\Sigma_t$ with respect to
$\mathbf n$. Taking $f=|\nabla u|^3$, we obtain
\[
A'(t)
=
\int_{\Sigma_t}
\left(
\frac{\langle \nabla |\nabla u|^3,\nabla u\rangle}{|\nabla u|^2}
+
|\nabla u|^2 H_{\Sigma_t}
\right)\,\di \sigma .
\]
Since
\[
\langle \nabla |\nabla u|^3,\nabla u\rangle
=
3|\nabla u|^3\Hess u(\mathbf n,\mathbf n),
\]
and, using $\Delta u=0$,
\[
H_{\Sigma_t}
=
-\frac{\Hess u(\mathbf n,\mathbf n)}{|\nabla u|},
\]
we get
\begin{align}\label{eq:first-derivative-parabolic}
A'(t)
=
2\int_{\Sigma_t}|\nabla u|\Hess u(\mathbf n,\mathbf n)\,\di \sigma .
\end{align}
Equivalently, since
\[
\langle \nabla |\nabla u|^2,\mathbf n\rangle
=
2\Hess u(\nabla u,\mathbf n)
=
2|\nabla u|\Hess u(\mathbf n,\mathbf n),
\]
we may rewrite \eqref{eq:first-derivative-parabolic} as
\begin{align}\label{eq:first-derivative-parabolic-variant}
A'(t)
=
\int_{\Sigma_t}
\langle \nabla |\nabla u|^2,\mathbf n\rangle\,\di \sigma .
\end{align}

Now let
\[
    \Gamma_{t,s}:=\partial M\cap \{t\le u\le s\}.
\]
The boundary of $\Omega_{t,s}$ consists of
\[
    \partial \Omega_{t,s}
    =
    \Sigma_s\cup \Sigma_t\cup \Gamma_{t,s},
\]
where the outward unit normal is $\mathbf n$ along $\Sigma_s$, $-\mathbf n$ along
$\Sigma_t$, and $\nu$ along $\Gamma_{t,s}$. Applying the divergence theorem to
$\nabla|\nabla u|^2$ gives
\begin{align}\label{eq:laplacian-by-parts}
\int_{\Omega_{t,s}}\Delta |\nabla u|^2\,\di V
&=
\int_{\Sigma_s}\langle \nabla |\nabla u|^2,\mathbf n\rangle\,\di \sigma
-
\int_{\Sigma_t}\langle \nabla |\nabla u|^2,\mathbf n\rangle\,\di \sigma 
+
\int_{\Gamma_{t,s}}\partial_\nu |\nabla u|^2\,\di \sigma .
\end{align}
Using \eqref{eq:first-derivative-parabolic-variant}, this becomes
\[
A'(s)-A'(t)
=
\int_{\Omega_{t,s}}\Delta |\nabla u|^2\,\di V
-
\int_{\Gamma_{t,s}}\partial_\nu |\nabla u|^2\,\di \sigma .
\]
By the Bochner formula and the harmonicity of $u$,
\[
\Delta |\nabla u|^2
=
2|\Hess u|^2+2\Ric(\nabla u,\nabla u).
\]
Therefore
\[
A'(s)-A'(t)
=
2\int_{\Omega_{t,s}}
\left(
|\Hess u|^2+\Ric(\nabla u,\nabla u)
\right)\,\di V
-
\int_{\Gamma_{t,s}}\partial_\nu |\nabla u|^2\,\di \sigma .
\]

It remains to compute the boundary term. Along $\partial M\setminus K$, recall that $\nabla u$ is tangent to $\partial M$ by the Neumann condition. Hence
\[
\partial_\nu |\nabla u|^2
=
2\Hess u(\nu,\nabla u)=-2h(\nabla u,\nabla u).
\]
Substituting this into the previous identity gives
\begin{align*}
A'(s)-A'(t)
&=
2\int_{\Omega_{t,s}}
\left(
|\Hess u|^2+\Ric(\nabla u,\nabla u)
\right)\,\di V 
+
2\int_{\Gamma_{t,s}}
h(\nabla u,\nabla u)\,dA .
\end{align*}
This proves \eqref{eq:derivative-of-A'-parabolic}.

If $\Ric_g\ge 0$ and $h\ge 0$, then $A'(s)-A'(t)\ge 0$
for any positive regular values $0<t<s$. Hence $A'(t)$ is monotone nondecreasing on the set of positive regular values. If, in addition,
\[
    \lim_{t\to\infty}A'(t)=0,
\]
then monotonicity implies $A'(t)\le 0$ for every positive regular value $t$. Therefore
$A(t)$ is monotone nonincreasing on positive regular values. It follows $A(t)$ is nonnegative and bounded above.
\end{proof}

\subsection{Contradiction argument}
We finish the proof of Proposition \ref{prop:Parabolic confirmation of Li's conjecture} in this section. The following lemma will be useful to derive the lower bound for the growth of $A(t)$.

We will only use \eqref{eq:derivative-of-A'-parabolic} in an integral. 
For regular values $t$, set
\[
    F(t):=A'(t)
    =
    \int_{\Sigma_t}
    \langle \nabla |\nabla u|^2,\mathbf n\rangle\,\di \sigma .
\]
Under the assumptions $\Ric\ge 0$ and $h\ge 0$, the identity
\eqref{eq:derivative-of-A'-parabolic} implies that $F$ is nondecreasing on the set of regular values. 
After redefining $F$ on the set of critical values, we regard $F$ as a right-continuous nondecreasing function on $[T,\infty)$. 
Its distributional derivative is therefore a nonnegative Radon measure, denoted by $\di F$. 
All differential inequalities involving $F'$ below are understood in this sense.

\begin{lemma}\label{ODE-lemma}
Let $A:[T,\infty)\to (0,\infty)$ be locally absolutely continuous, and let
$F:[T,\infty)\to \mathbb R$ be a right-continuous nondecreasing function such that
\[
    A'(t)=F(t)
\]
for a.e. $t\ge T$. Suppose that $F(t)>0$ for all $t\ge T$. Assume moreover that, for some constant $C>1$, the inequality
\[
    F(t)^2\,\di t \le C A(t)\,\di F(t)
\]
holds as an inequality of Radon measures on $[T,\infty)$. Then there exists a constant $c>0$ such that
\[
    A(t)\ge c\, t^{\frac{C}{C-1}}
\]
for all sufficiently large $t$.
\end{lemma}

\begin{proof}
Since $F$ is nondecreasing and right-continuous, $\di F$ is a nonnegative Radon measure. The assumed measure inequality gives
\[
    F^2\,\di t \le C A\,\di F .
\]
Equivalently,
\[
    A^{-1/C}\,\di F
    \ge
    \frac{1}{C}A^{-1-\frac1C}F^2\,\di t .
\]
We now compute in the sense of distributions. Since $A$ is locally absolutely continuous and $A'=F$ a.e., we have $\di A=F\,\di t$.
Using the product rule for BV functions, we obtain
\[
\begin{aligned}
    \di\left(F A^{-1/C}\right)
    &=
    A^{-1/C}\,\di F
    +
    F\,\di\left(A^{-1/C}\right)  \\
    &=
    A^{-1/C}\,\di F
    -
    \frac{1}{C}F A^{-1-\frac1C}\,\di A  \\
    &=
    A^{-1/C}\,\di F
    -
    \frac{1}{C}A^{-1-\frac1C}F^2\,\di t.
\end{aligned}
\]
By the measure inequality above, this gives
\[
    \di\left(F A^{-1/C}\right)\ge 0.
\]
Hence
\[
    F(t)A(t)^{-1/C}
\]
is nondecreasing on $[T,\infty)$. In particular, for every $t\ge T$,
\[
    F(t)A(t)^{-1/C} \ge F(T)A(T)^{-1/C}.
\]
Set
\[
    c_0:=F(T)A(T)^{-1/C}>0.
\]
Then
\[
    F(t)\ge c_0 A(t)^{1/C}.
\]
Since $A'=F$ a.e., it follows that
\[
    A'(t)\ge c_0 A(t)^{1/C}
\]
for a.e. $t\ge T$.

Therefore,
\[
\frac{\di}{\di t}A(t)^{1-\frac1C}
=
\left(1-\frac1C\right)A(t)^{-\frac1C}A'(t)
\ge
\left(1-\frac1C\right)c_0
\]
for a.e. $t\ge T$. Integrating from $T$ to $t$, we obtain
\[
    A(t)^{1-\frac1C}
    \ge
    A(T)^{1-\frac1C}
    +
    \left(1-\frac1C\right)c_0(t-T).
\]
Thus, for all sufficiently large $t$, there exists $c_1>0$ such that
\[
    A(t)^{1-\frac1C}\ge c_1 t.
\]
Raising both sides to the power $C/(C-1)$ yields
\[
    A(t)\ge c\,t^{\frac{C}{C-1}}
\]
for some constant $c>0$ and all sufficiently large $t$. This proves the lemma.
\end{proof}

\begin{proof} [Proof of Proposition \ref{prop:Parabolic confirmation of Li's conjecture}]
We split our proof into two cases depending on the sign of the quantity $F(t)=A'(t)$. The treatments will be different in each case.

\textbf{Case I.} Suppose that there exists a regular value $t_0$ such that
\[
    F(t_0)=A'(t_0)>0.
\]
Since $F$ is nondecreasing, after replacing $t_0$ by a larger regular value if necessary, we may choose $T\ge t_0$ such that
\[
    F(t)>0
\]
for all $t\ge T$.

We now derive a differential inequality for $F$ in the weak sense. For a regular value $t$, the first variation formula gives
\[
    F(t)=A'(t)
    =
    2\int_{\Sigma_t} |\nabla u|\,\nabla^2u(\mathbf n,\mathbf n)\,\di\sigma .
\]
By the Cauchy--Schwarz inequality,
\[
\begin{aligned}
    F(t)^2
    &\le
    4
    \left(
    \int_{\Sigma_t}|\nabla u|^3\,\di\sigma
    \right)
    \left(
    \int_{\Sigma_t}
    \frac{|\nabla^2u(\mathbf n,\mathbf n)|^2}{|\nabla u|}\,\di\sigma
    \right)  \\
    &=
    4A(t)
    \int_{\Sigma_t}
    \frac{|\nabla^2u(\mathbf n,\mathbf n)|^2}{|\nabla u|}\,\di\sigma .
\end{aligned}
\]
Since
\[
    |\nabla^2u(\mathbf n,\mathbf n)|^2
    \le
    |\nabla |\nabla u||^2,
\]
and the improved Kato inequality for harmonic functions gives
\[
    |\nabla^2u|^2
    \ge
    \frac{n}{n-1}|\nabla |\nabla u||^2,
\]
we obtain
\[
    F(t)^2
    \le
    \frac{4(n-1)}{n}A(t)
    \int_{\Sigma_t}
    \frac{|\nabla^2u|^2}{|\nabla u|}\, \di\sigma .
\]
On the other hand, the monotonicity identity implies, in the sense of Radon measures,
\[
    \di F
    \ge
    2\left(\int_{\Sigma_t}\frac{|\nabla^2u|^2}{|\nabla u|}\,\di \sigma
    \right)\di t .
\]
Indeed, this follows from the coarea formula applied to the nonnegative bulk term
\[
    2\int_{\{t\le u\le s\}} |\nabla^2u|^2\,\di V
\]
appearing in the monotonicity formula, while the Ricci and boundary terms are nonnegative under the assumptions $\Ric\ge0$ and $h\ge0$. Therefore,
\[
    F(t)^2\,\di t
    \le
    \frac{2(n-1)}{n}A(t)\,\di F(t)
\]
as an inequality of Radon measures on $[T,\infty)$. Then Lemma~\ref{ODE-lemma} gives
\begin{equation}\label{eq:parabolic-lower-bound}
    A(t)\ge c\,t^{\frac{C}{C-1}}
    =
    c\,t^{\frac{2n-2}{n-2}}
\end{equation}
for all sufficiently large $t$.
However by the Cheng--Yau gradient estimate on sufficiently large level sets,
\[
    \sup_{\Sigma_t}|\nabla u|^2\le C t^2.
\]
Moreover, by the divergence theorem, the flux
\[
    \int_{\Sigma_t}|\nabla u|\,\di \sigma
\]
is independent of $t$. Therefore
\[
    A(t)
    =\int_{\Sigma_t}|\nabla u|^3\,\di \sigma  
    \le\sup_{\Sigma_t}|\nabla u|^2\int_{\Sigma_t}|\nabla u|\,\di \sigma  
    \le C t^2 .
\]
This contradicts \eqref{eq:parabolic-lower-bound} as $t\to\infty$. 

\textbf{Case II.} We have shown that
\[
    F(t)=A'(t)\le 0
\]
for all sufficiently large regular values $t$. Since $A'(t)$ is monotone nondecreasing and
$A(t)\ge 0$, it follows that
\[
    \lim_{t\to\infty}A'(t)=0.
\]
Indeed, if $\lim_{t\to\infty}A'(t)=\ell<0$, then $A(t)$ would tend to $-\infty$, contradicting the nonnegativity of $A$.

Fix $T>0$ sufficiently large and regular. Letting $s\to\infty$ in the monotonicity identity gives
\[
2\int_{\{u\ge T\}}
\left(
|\Hess u|^2+\Ric(\nabla u,\nabla u)
\right)\,\di V
+
2\int_{\partial M\cap \{u\ge T\}}
h(\nabla u,\nabla u)\,\di A\le -A'(T)<\infty .
\]
In particular, using $\Ric\ge0$ and $h\ge1$, we obtain
\[
    \int_{\{u\ge T\}}|\Hess u|^2\,\di V
    +
    \int_{\partial M\cap\{u\ge T\}}|\nabla u|^2\,d\sigma
    <\infty .
\]

Choose a smooth cutoff function $\chi:\mathbb R\to[0,1]$ such that
\[
    \chi(x)=0 \quad \text{for } x\le T,
    \qquad
    \chi(x)=1 \quad \text{for } x\ge T+1.
\]
Set
\[
    \phi:=\chi(u)u.
\]
Applying the trace Poincaré inequality Theorem~\ref{Trace Poincare inequality} to $\phi$, we get
\[
\int_{\{u\ge T+1\}}|\nabla u|^2\,\di V
\le
\int_M |\nabla \phi|^2\,\di V  \\
\le
\frac{2}{n}\int_{\partial M}|\nabla \phi|^2\,\di \sigma
+
\int_M |\nabla |\nabla \phi||^2\,\di V .
\]
Since $ |\nabla |\nabla \phi||^2\le |\Hess \phi|^2$, we obtain
\[
\int_{\{u\ge T+1\}}|\nabla u|^2\,\di V
\le
\frac{2}{n}\int_{\partial M}|\nabla(\chi(u)u)|^2\,\di \sigma
+
\int_M |\Hess(\chi(u)u)|^2\,\di V .
\]
On $\partial M$, the Neumann condition gives $\partial_\nu u=0$, hence
\[
    |\nabla(\chi(u)u)|^2
    \le C(T)\chi(u)^2|\nabla u|^2
        +C(T)|\chi'(u)|^2|\nabla u|^2 .
\]
Moreover, since $\chi'$ and $\chi''$ are supported in $\{T\le u\le T+1\}$, the Hessian expansion gives
\[
    |\Hess(\chi(u)u)|^2
    \le
    C(T)\left(
    |\Hess u|^2
    +
    \mathbf 1_{\{T\le u\le T+1\}}
    \left(|\nabla u|^2+|\nabla u|^4\right)
    \right).
\]
Therefore
\[
\begin{aligned}
\int_{\{u\ge T+1\}}|\nabla u|^2\,\di V
&\le
C(n,T)
\Bigg(
\int_{\{u\ge T\}}|\Hess u|^2\,\di V
+
\int_{\partial M\cap\{u\ge T\}}|\nabla u|^2\,\di \sigma  \\
&\qquad\qquad
+
\int_{\{T\le u\le T+1\}}
\left(|\nabla u|^2+|\nabla u|^4\right)\,\di V
\Bigg)
<\infty .
\end{aligned}
\]
Here the last integral is finite because the slab $\{T\le u\le T+1\}$ is compact and lies in the smooth region $M\setminus K$.

On the other hand, by the coarea formula,
\[
\int_{\{u\ge T+1\}}|\nabla u|^2\,\di V
=
\int_{T+1}^{\infty}
\left(
\int_{\Sigma_t}|\nabla u|\,\di \sigma
\right)\,\di t =\int_{T+1}^{\infty} c\,\di t
=
\infty,
\]
 It contradicts the finiteness obtained above.
\end{proof}

\section{Nonparabolic case}\label{sec:nonparabolic}
In this section we treat the nonparabolic case. 
The relevant function for the construction of the monotone quantity is the Neumann positive Green function. 
After fixing a pole $p\in \mathring{M}$, we write this Green function as $G=G_N(p,\cdot)$ and introduce the associated Green distance function
\[
    b:=G^{\frac{1}{2-n}}.
\]
The level sets of $b$ play the same role in the nonparabolic case as the level sets of the Evans potential in the parabolic case. 
The main goal of this section is to construct a scale-invariant monotone quantity associated with these level sets and then use it, together with the trace Poincar\'e inequality, to rule out the nonparabolic case.

\begin{proposition}\label{prop:Nonparabolic confirmation of Li's conjecture}
    Let $(M^n,g)$ be a complete Riemannian $n$-manifold with nonempty boundary $\partial M$ and has
    \begin{equation}
        \Ric_g\ge 0, \quad h\ge 1.
    \end{equation}
    Then $(M,g)$ cannot be noncompact and nonparabolic. 
\end{proposition}

The rest of this section is devoted to the proof of Proposition~\ref{prop:Nonparabolic confirmation of Li's conjecture}. We will prove by contradiction, so in the following we always assume that $(M,g)$ is noncompact and nonparabolic.

We first recall the existence and properness of the minimal Neumann Green function under the present curvature assumptions.

\subsection{Existence and properness of the minimal Neumann Green function}

\begin{proposition}\label{prop:existence-properness-Neumann-Green}
Let $(M^n,g)$, $n\ge 3$, be a complete noncompact Riemannian manifold with smooth boundary $\partial M$. Assume that $M$ is nonparabolic 
then, for every pole $p\in \mathring{M}$, there exists a minimal positive Neumann Green function $ G=G_N(p,\cdot)$ satisfying
\[
    \Delta G=-c_n\delta_p
    \quad \text{in } M,
    \qquad
    \partial_\nu G=0
    \quad \text{on } \partial M,
    \qquad
    G>0
    \quad \text{on } M\setminus\{p\}.
\]
Moreover, if $\Ric_g\ge 0$ and $h\ge0$, then
\[
    G(p,x)\to 0
    \quad \text{as } x\to\infty.
\]
Consequently, the associated Green distance function
\[
    b(x):=G(p,x)^{\frac{1}{2-n}}
\]
is proper, that is,
\[
    b(x)\to\infty
    \quad \text{as } x\to\infty.
\]
\end{proposition}

\begin{proof}
The existence of the minimal positive Neumann Green function is standard for nonparabolic manifolds. 
Indeed, let $P_N(t,x,y)$ denote the Neumann heat kernel. 
Nonparabolicity is equivalent to the finiteness of the Green kernel
\[
    G_N(x,y):=\int_0^\infty P_N(t,x,y)\,\di t
\]
for $x\ne y$. 

The properness of $G_N$ under the curvature assumptions is also well-known. It follows from the Gaussian estimates of the heat kernel that 
\[
G_N(x,y)\le \int_{\dist_g(x,y)}^\infty \frac{s}{V(B(x,s))}\, \di s,
\]
see Brue--Semola \cite{BS20}*{Proposition 2.3} and its proof. Note that by an observation of Han \cite{HanMeasureRigidity}, if $(M,g, \vol_g)$ satisfies $\Ric_g\ge0$ and $h\ge 0$ then it is an $\RCD(0,n)$ space, and by definition the Green's function in an $\RCD(0,n)$ space is the Neuamnn one, so \cite{BS20} is applicable. This estimate implies $G_N\to 0$ as $\dist_g(x,y)\to\infty$ which is in turn equivalent to
\[
    b(x)\to\infty
    \quad \text{as } x\to\infty.
\]
Thus $b$ is a proper function on $M$.
\end{proof}

In the sequel, we fix a pole $p\in \mathring{M}$ and write
\[
    G:=G_N(p,\cdot),
    \qquad
    b:=G^{\frac{1}{2-n}}.
\]
All computations are carried out on $M\setminus\{p\}$ and on regular level sets of $b$.

\subsection{Monotone functional}
We now turn to the monotonicity formula associated with minimal positive Neumann Green's function constructed above. 

Let $(M^n,g)$, $n\ge 3$, be a complete noncompact Riemannian manifold
with nonempty boundary $\del M$.
Assume that $M$ is nonparabolic for the Neumann problem, and let $G$ be
the minimal positive Neumann Green's function with pole $p\in \mathring{M}$,
normalized so that
\[
        \Delta G=-c_n\delta_p,
        \qquad
        \del_\nu G=0
        \quad \text{on } \del M .
\]
Throughout this subsection, define the Green distance function
\[
        b:=G^{\frac{1}{2-n}} .
\]
For a regular value $t>0$, set
\[
        \Sigma_t:=\{b=t\},
        \qquad
        \n:=\frac{\nabla b}{|\nabla b|}.
\]
For regular values $t$, define
\[
        A(t):=t^{1-n}\int_{\Sigma_t}|\nabla b|^3\,\di\sigma.
\]

\begin{lemma}\label{lem:green-distance-basic-identities}
Away from the pole $p$, it holds
\[
        b\Delta b=(n-1)|\nabla b|^2,
        \qquad
        \Delta b^2=2n|\nabla b|^2 .
\]
Moreover,
\[
        \del_\nu b=0
        \quad\text{on }\del M.
\]
Consequently, $\nabla b$ is tangent to $\del M$, and every regular level set
$\Sigma_t=\{b=t\}$ meets $\del M$ orthogonally.
\end{lemma}

\begin{proof}
Away from $p$, $G=b^{2-n}$ is harmonic. Hence
\[
        0=\Delta b^{2-n}
        =
        (2-n)b^{1-n}\Delta b
        +
        (2-n)(1-n)b^{-n}|\nabla b|^2,
\]
which gives
\[
        b\Delta b=(n-1)|\nabla b|^2.
\]
Then
\[
        \Delta b^2=2b\Delta b+2|\nabla b|^2=2n|\nabla b|^2.
\]
Since $\del_\nu G=0$ on $\del M$, we also have $\del_\nu b=0$ on $\del M$.
Thus $\nabla b$ is tangent to $\del M$. Along
$\Gamma_t:=\Sigma_t\cap\del M$, the normal $\n$ to $\Sigma_t$ is tangent to
$\del M$, and hence $\Sigma_t$ meets $\del M$ orthogonally.
\end{proof}

\begin{lemma}\label{lem:first-variation-nonparabolic}
For every regular value $t$ of $b$, it holds
\[
        A'(t)
        =
        2t^{1-n}
        \int_{\Sigma_t}
        |\nabla b|\,\Hess b(\n,\n)\,\di\sigma .
\]
Equivalently,
\[
        A'(t)
        =
        t^{1-n}
        \int_{\Sigma_t}
        \del_{\n}|\nabla b|^2\,\di\sigma .
\]
\end{lemma}

\begin{proof}
Let $I(t):=\int_{\Sigma_t}|\nabla b|^3\,\di\sigma$, so that
$A(t)=t^{1-n}I(t)$. For a level set integral,
\[
        \frac{d}{\di t}\int_{\Sigma_t} f\,\di\sigma
        =
        \int_{\Sigma_t}
        \left(
        \frac{\del_{\n}f}{|\nabla b|}
        +
        \frac{fH}{|\nabla b|}
        \right)\,\di\sigma,
\]
where $H$ is the mean curvature of $\Sigma_t$ with respect to $\n$.
Taking $f=|\nabla b|^3$, we get
\[
        \frac{\del_{\n}(|\nabla b|^3)}{|\nabla b|}
        =
        3|\nabla b|\Hess b(\n,\n).
\]
Since $H=\divsymb_{\Sigma_t}\n$, \Cref{lem:green-distance-basic-identities}
gives, on $\Sigma_t$,
\[
        H
        =
        \frac{\Delta b}{|\nabla b|}
        -
        \frac{\Hess b(\n,\n)}{|\nabla b|}
        =
        \frac{n-1}{t}|\nabla b|
        -
        \frac{\Hess b(\n,\n)}{|\nabla b|}.
\]
Hence
\[
        \frac{|\nabla b|^3H}{|\nabla b|}
        =
        \frac{n-1}{t}|\nabla b|^3
        -
        |\nabla b|\Hess b(\n,\n).
\]
Thus
\[
        I'(t)
        =
        \int_{\Sigma_t}
        \left[
        2|\nabla b|\Hess b(\n,\n)
        +
        \frac{n-1}{t}|\nabla b|^3
        \right]\di\sigma .
\]
Differentiating $A(t)=t^{1-n}I(t)$ gives
\[
        A'(t)
        =
        (1-n)t^{-n}I(t)+t^{1-n}I'(t).
\]
The zeroth-order terms cancel, so
\[
        A'(t)
        =
        2t^{1-n}
        \int_{\Sigma_t}
        |\nabla b|\Hess b(\n,\n)\,\di\sigma .
\]
The second formula follows from
\[
        \del_{\n}|\nabla b|^2
        =
        2|\nabla b|\Hess b(\n,\n).
\]
\end{proof}

\begin{lemma}\label{lem:trace-free-hessian-expression}
Let
\[
        B:=\Hess b^2-\frac{\Delta b^2}{n}g= B=\Hess b^2-2|\nabla b|^2g .
\]
Then
\[
        B(\n,\n)=2b\Hess b(\n,\n).
\]
Consequently, for every regular value $t$,
\[
        A'(t)
        =
        t^{-n}
        \int_{\Sigma_t}
        |\nabla b|B(\n,\n)\,\di\sigma .
\]
\end{lemma}

\begin{proof}
Since
\[
        \Hess b^2=2b\,\Hess b+2\,db\otimes db,
\]
we have
\[
        B(\n,\n)
        =
        2b\Hess b(\n,\n)+2|\nabla b|^2-2|\nabla b|^2
        =
        2b\Hess b(\n,\n).
\]
Combining this with \Cref{lem:first-variation-nonparabolic}, and using
$b=t$ on $\Sigma_t$, gives the formula for $A'(t)$.
\end{proof}

\begin{lemma}\label{lem:weighted-interior-divergence}
Away from the pole $p$, we have
\[
        \divsymb\left(b^{4-2n}\nabla|\nabla b|^2\right)
        =
        \frac12 b^{2-2n}
        \left(
        |B|^2+\Ric(\nabla b^2,\nabla b^2)
        \right).
\]
\end{lemma}

\begin{proof}
By Bochner's formula,
\[
        \Delta |\nabla b|^2
        =
        2|\Hess b|^2
        +2\Ric(\nabla b,\nabla b)
        +2\lp \nabla b,\nabla\Delta b\rp .
\]
Using $\Delta b=(n-1)b^{-1}|\nabla b|^2$, we get
\[
        \lp \nabla b,\nabla\Delta b\rp
        =
        (n-1)
        \left(
        \frac{\lp \nabla b,\nabla |\nabla b|^2\rp}{b}
        -
        \frac{|\nabla b|^4}{b^2}
        \right).
\]
Therefore
\[
\begin{aligned}
        \divsymb\left(b^{4-2n}\nabla |\nabla b|^2\right)
        &=
        b^{4-2n}\Delta |\nabla b|^2
        +(4-2n)b^{3-2n}\lp \nabla b,\nabla |\nabla b|^2\rp       \\
        &=
        2b^{4-2n}|\Hess b|^2
        +2b^{4-2n}\Ric(\nabla b,\nabla b)                 \\
        &\quad
        +2b^{3-2n}\lp \nabla b,\nabla |\nabla b|^2\rp
        -2(n-1)b^{2-2n}|\nabla b|^4 .
\end{aligned}
\]
On the other hand,
\[
        B=2b\Hess b+2db\otimes db-2|\nabla b|^2g.
\]
A direct expansion gives
\[
        |B|^2
        =
        4b^2|\Hess b|^2
        +4b\lp \nabla b,\nabla |\nabla b|^2\rp
        -4(n-1)|\nabla b|^4 .
\]
Since
\[
        \Ric(\nabla b^2,\nabla b^2)=4b^2\Ric(\nabla b,\nabla b),
\]
the identity follows.
\end{proof}

\begin{lemma}\label{lem:free-boundary-contribution}
Along $\del M$,
\[
        \del_\nu|\nabla b|^2
        =
        -2h(\nabla b,\nabla b)
        =
        -2|\nabla b|^2h(\n,\n).
\]
Consequently, on
\[
        \Gamma_{t_1,t_2}:=\del M\cap\{t_1\le b\le t_2\},
\]
we have
\[
        -\int_{\Gamma_{t_1,t_2}}
        b^{4-2n}\del_\nu|\nabla b|^2\,\di \sigma
        =
        2\int_{\Gamma_{t_1,t_2}}
        b^{4-2n}|\nabla b|^2h(\n,\n)\,\di \sigma .
\]
\end{lemma}

\begin{proof}
By \Cref{lem:green-distance-basic-identities}, $\nabla b$ is tangent to
$\del M$. If $X$ is tangent to $\del M$, differentiating the Neumann condition
gives
\[
        0=X\lp \nabla b,\nu\rp
        =
        \Hess b(X,\nu)+\lp \nabla b,\nabla_X\nu\rp .
\]
By the convention fixed in the introduction, taking $X=\nabla b$ gives
\[
        \Hess b(\nabla b,\nu)=-h(\nabla b,\nabla b).
\]
Therefore
\[
        \del_\nu|\nabla b|^2
        =
        2\Hess b(\nabla b,\nu)
        =
        -2h(\nabla b,\nabla b)
        =
        -2|\nabla b|^2h(\n,\n),
\]
and the integral identity follows.
\end{proof}

\begin{theorem}\label{thm:weighted-second-variation-nonparabolic}
Let $0<t_1<t_2$ be regular values of $b$, and set
\[
        \Omega_{t_1,t_2}:=\{t_1\le b\le t_2\}.
\]
Then
\[
\begin{aligned}
        t_2^{3-n}A'(t_2)-t_1^{3-n}A'(t_1)
        &=
        \frac12
        \int_{\Omega_{t_1,t_2}}
        b^{2-2n}
        \left(
        |B|^2+\Ric(\nabla b^2,\nabla b^2)
        \right)\,\di V                                      \\
        &\quad
        +2\int_{\del M\cap\Omega_{t_1,t_2}}
        b^{4-2n}|\nabla b|^2
        h(\n,\n)\,\di \sigma .
\end{aligned}
\]
In particular, if $\Ric\ge0$ and $h\ge0$, then
\[
        t^{3-n}A'(t)
\]
is monotone nondecreasing in $t$. If in addition $h\ge1$, then
\begin{align*}
        t_2^{3-n}A'(t_2)-t_1^{3-n}A'(t_1)\ge
        \frac12
        \int_{\Omega_{t_1,t_2}}
        b^{2-2n}|B|^2\,\di V 
        +2\int_{\del M\cap\Omega_{t_1,t_2}}
        b^{4-2n}|\nabla b|^2\,\di \sigma .
\end{align*}
At regular values where the derivative exists,
\begin{align*}
        \frac{d}{\di t}\left(t^{3-n}A'(t)\right)=
        \frac12
        t^{2-2n}
        \int_{\Sigma_t}
        \frac{|B|^2+\Ric(\nabla b^2,\nabla b^2)}{|\nabla b|}
        \,\di\sigma        
        +
        2t^{4-2n}
        \int_{\del\Sigma_t}
        |\nabla b|h(\n,\n)\,\di\ell .
\end{align*}
If $\Ric\ge0$ and $h\ge1$, then
\begin{align*}
        \frac{d}{\di t}\left(t^{3-n}A'(t)\right)\ge
        \frac12
        t^{2-2n}
        \int_{\Sigma_t}
        \frac{|B|^2}{|\nabla b|}
        \,\di\sigma
        +
        2t^{4-2n}
        \int_{\del\Sigma_t}
        |\nabla b|\,\di\ell .
\end{align*}
Equivalently, by the coarea formula on $\del M$,
\begin{align*}
        t_2^{3-n}A'(t_2)-t_1^{3-n}A'(t_1)
        \ge
        \frac12
        \int_{\Omega_{t_1,t_2}}
        b^{2-2n}|B|^2\,\di V   
        +
        2\int_{t_1}^{t_2}
        s^{4-2n}
        \int_{\del\Sigma_s}|\nabla b|\,\di\ell\,\di s .
\end{align*}
\end{theorem}

\begin{proof}
By \Cref{lem:first-variation-nonparabolic}, we know that
\[
        t^{3-n}A'(t)
        =
        \int_{\Sigma_t}
        b^{4-2n}\del_{\n}|\nabla b|^2\,\di\sigma .
\]
We now apply the divergence theorem to the vector field
\[
        b^{4-2n}\nabla|\nabla b|^2
\]
on $\Omega_{t_1,t_2}$. The outward normal is $\n$ on $\Sigma_{t_2}$,
$-\n$ on $\Sigma_{t_1}$, and $\nu$ on $\del M\cap\Omega_{t_1,t_2}$. Hence
\[
\begin{aligned}
        t_2^{3-n}A'(t_2)-t_1^{3-n}A'(t_1)=
        \int_{\Omega_{t_1,t_2}}
        \divsymb\left(b^{4-2n}\nabla|\nabla b|^2\right)\,\di V 
        -
        \int_{\del M\cap\Omega_{t_1,t_2}}
        b^{4-2n}\del_\nu|\nabla b|^2\,\di \sigma .
\end{aligned}
\]
The interior term is given by \Cref{lem:weighted-interior-divergence}, and the
boundary term is given by \Cref{lem:free-boundary-contribution}. This proves
the identity. The monotonicity and quantitative lower bounds follow from
$\Ric\ge0$, $h\ge0$, and the stronger assumption $h\ge1$. The differential and
coarea forms follow from the coarea formula in $M$ and on $\del M$, using
$|\nabla^{\del M}b|=|\nabla b|$ along $\del M$.
\end{proof}

As in the parabolic case, the monotonicity formula will be used only in an integral. For regular values $t$, define
\[
        F(t):=t^{3-n}A'(t).
\]
Under the assumptions $\Ric\ge0$ and $h\ge0$,
\Cref{thm:weighted-second-variation-nonparabolic} implies that $F$ is
nondecreasing on the set of regular values. After redefining $F$ on the set of
critical values, we regard $F$ as a right-continuous nondecreasing function on
$[T,\infty)$. Its distributional derivative is then a nonnegative Radon
measure, denoted by $\di F$.

When $h\ge1$, the quantitative part of
\Cref{thm:weighted-second-variation-nonparabolic}, together with the coarea
formula, gives the following inequality of Radon measures:
\begin{equation}\label{eq:finite-energy}
      \di F\ge
        \frac12 t^{2-2n}
        \left(
        \int_{\Sigma_t}\frac{|B|^2}{|\nabla b|}\,\di\sigma
        \right)\di t
        +
        2t^{4-2n}
        \left(
        \int_{\partial\Sigma_t}|\nabla b|\,\di\ell
        \right)\di t .
\end{equation}
Thus all later occurrences of $F'$ are to be understood in this measure sense.


\subsection{Contradiction argument}
In this section we complete the proof of Proposition \ref{prop:Nonparabolic confirmation of Li's conjecture} which 
confirms Li's Conjecture in the nonparabolic case. As in the parabolic case we start with an ODE lemma.

\begin{lemma}\label{lem:weighted-growth-nonparabolic}
Let $A:[T,\infty)\to(0,\infty)$ be locally absolutely continuous, and let
$F:[T,\infty)\to\mathbb R$ be a right-continuous nondecreasing function. Assume that
\begin{align}\label{weak-interpretation-nonparabolic}
        A'(t)=t^{n-3}F(t)
        \quad\text{for a.e. }t\ge T,
\end{align}
and the inequality
\[
        F\,\di A\le 2A\,\di F
\]
holds as an inequality of Radon measures on $[T,\infty)$. If $F(T)>0$, then there exists a constant $c>0$ such that
\[
        A(t)\ge c\,t^{2n-4}
\]
for all sufficiently large $t$.
\end{lemma}

\begin{proof}
Since $F$ is nondecreasing and $F(T)>0$, we have $F(t)\ge F(T)>0$ for all
$t\ge T$. In particular, $\di F$ is a nonnegative Radon measure.

Using the product rule for a function of bounded variation multiplied by a
locally absolutely continuous function, we compute in the sense of distributions:
\[
        \di \left(FA^{-1/2}\right)
        =
        A^{-1/2}\,\di F
        +
        F\,\di\left(A^{-1/2}\right) 
        =
        A^{-1/2}\,\di F
        -
        \frac12 FA^{-3/2}\,\di A .
\]
The assumed measure inequality
\[
        F\,\di A\le 2A\, \di F
\]
implies
\[
        \frac12 FA^{-3/2}\,\di \sigma
        \le
        A^{-1/2}\,\di F.
\]
Therefore
\[
        \di\left(FA^{-1/2}\right)\ge0.
\]
Hence $FA^{-1/2}$ is nondecreasing on $[T,\infty)$. Consequently, for every
$t\ge T$,
\[
        F(t)A(t)^{-1/2}
        \ge
        F(T)A(T)^{-1/2}.
\]
Set
\[
        c_0:=F(T)A(T)^{-1/2}>0.
\]
Then
\[
        F(t)\ge c_0 A(t)^{1/2}
        \quad\text{for all }t\ge T.
\]
Since $A'(t)=t^{n-3}F(t)$ for a.e. $t\ge T$, it follows that
\[
        A'(t)\ge c_0 t^{n-3}A(t)^{1/2}
\]
for a.e. $t\ge T$. Hence
\[
        \frac{d}{\di t}A(t)^{1/2}
        =
        \frac{A'(t)}{2A(t)^{1/2}}
        \ge
        \frac{c_0}{2}t^{n-3}
\]
for a.e. $t\ge T$.

Integrating from $T$ to $t$, we obtain
\[
        A(t)^{1/2} \ge
        A(T)^{1/2} +\frac{c_0}{2}\int_T^t s^{n-3}\,\di s \ge A(T)^{1/2}
        +c\,t^{n-2}.
\]
Therefore
\[
        A(t)^{1/2}\ge c\,t^{n-2}
\]
for all sufficiently large $t$, and it in turn gives
\[
        A(t)\ge c\,t^{2n-4}.
\]
This proves the lemma.
\end{proof}

\begin{proof}[Proof of Proposition \ref{prop:Nonparabolic confirmation of Li's conjecture} ]
    We split the proof into two cases, depending on signs of the quantity $F(t)=t^{n-3}A'(t)$, just like in the parabolic setting.  

  \textbf{Case I.}  Suppose that there exists a regular value $t_0$ such that $F(t_0)>0$.
        
 We first derive the measure inequality needed for \Cref{lem:weighted-growth-nonparabolic}. From
\Cref{lem:trace-free-hessian-expression} and the Cauchy--Schwarz inequality,
\[
\begin{aligned}
    \bigl(A'(t)\bigr)^2
    &=
    t^{-2n}
    \left(
        \int_{\Sigma_t} |\nabla b|B(\n,\n)\,\di\sigma
    \right)^2  \\
    &\leq
    t^{-2n}
    \left(
        \int_{\Sigma_t} |\nabla b|^3\,\di\sigma
    \right)
    \left(
        \int_{\Sigma_t}\frac{|B(\n,\n)|^2}{|\nabla b|}\,\di\sigma
    \right)  \\
    &\leq
    t^{-2n}
    \left(
        \int_{\Sigma_t} |\nabla b|^3\,\di\sigma
    \right)
    \left(
        \int_{\Sigma_t}\frac{|B|^2}{|\nabla b|}\,\di\sigma
    \right).
\end{aligned}
\]
Since
\[
    \int_{\Sigma_t}|\nabla b|^3\,\di\sigma=t^{n-1}A(t),
\]
we obtain
\[
    \bigl(A'(t)\bigr)^2
    \leq
    t^{-n-1}A(t)
    \int_{\Sigma_t}\frac{|B|^2}{|\nabla b|}\,\di\sigma .
\]
Using the weak interpretation from Theorem \ref{thm:weighted-second-variation-nonparabolic},
\[
    \di F
    \geq
    \frac12 t^{2-2n}
    \int_{\Sigma_t}\frac{|B|^2}{|\nabla b|}\,\di\sigma\,\di t,
    \qquad
    F(t):=t^{3-n}A'(t),
\]
we conclude that
\[
    t^{3-n}\bigl(A'(t)\bigr)^2\,\di t
    \leq
    2A(t)\,\di F(t)
\]
as an inequality of Radon measures.
Since $F(t)=t^{3-n}A'(t)$ and $\di A=A'(t)\,\di t$, this is equivalent to
\[
    F\,\di A\le 2A\,\di F.
\]

Now suppose that there exists a sufficiently large $t_0$ such that
$F(t_0)>0$. Since $F$ is nondecreasing, $F>0$ on $[t_0,\infty)$. Applying
\Cref{lem:weighted-growth-nonparabolic} on $[t_0,\infty)$, we obtain
\[
    A(t)\geq ct^{2n-4}.
\]

On the other hand, the Cheng--Yau gradient estimate gives $|\nabla b|\leq C$
on large level sets $\Sigma_t$. 
Moreover,
\[
    t^{1-n}\int_{\Sigma_t}|\nabla b|\,\di\sigma=c,
\]
by Colding \cite{Colding2012}. Using the gradient bound, we get
\[
    A(t)
    =
    t^{1-n}\int_{\Sigma_t}|\nabla b|^3\,\di\sigma 
    \leq
    \left(\sup_{\Sigma_t}|\nabla b|^2\right)
    t^{1-n}\int_{\Sigma_t}|\nabla b|\,\di\sigma  
    \leq
    C.
\]
This contradicts the lower bound
\[
    A(t)\geq ct^{2n-4}
\]
as $t\to\infty$.

Therefore $F(t_0)>0$ is impossible.

  \textbf{Case II.} Now that we showed $F(t)>0$ is impossible for large enough $t$, then there exists $T>0$ sufficiently large so that
\[
        F(t)\leq 0
        \quad\text{for all } t\geq T .
\]
Since $F$ is nondecreasing, the limit
\[
        \ell:=\lim_{t\to\infty}F(t)
\]
exists and satisfies $\ell\leq0$. It is clear that $\ell=0$ otherwise, $\ell<0$ and
\[
A(t)=\int_T^t t^{n-3}F(t)
\]
will decay to a negative value, a contradiction. Therefore,
\[
        \lim_{t\to\infty}F(t)=0.
\]

Recall that from Theorem~\ref{thm:weighted-second-variation-nonparabolic}, $F$ is right-continuous and nondecreasing. The measure estimates  \eqref{eq:finite-energy} holds:
\begin{equation}\label{eq:finite-engergy2}
     \di F
        \ge
        \frac12 t^{2-2n}
        \left(
        \int_{\Sigma_t}\frac{|B|^2}{|\nabla b|}\,\di\sigma
        \right)\di t
        +
        2t^{4-2n}
        \left(
        \int_{\partial\Sigma_t}|\nabla b|\,\di\ell
        \right)\di t .
\end{equation}

We next derive a scale-invariant weighted energy estimates. Since $F$ is monotone nondecreasing and converges to $0$, the measure $\di F$ has finite mass on $[T,\infty)$. 
Multiplying \eqref{eq:finite-engergy2} by
$t^{n-2}$ and integrating over $[T,R]$, we obtain
\[
        \int_{[T,R]} t^{n-2}\,\di F(t)
        \geq
        \frac12\int_T^R t^{-n}
        \left(
        \int_{\Sigma_t}\frac{|B|^2}{|\nabla b|}\,\di\sigma
        \right)\di t  
        +
        2\int_T^R t^{2-n}
        \left(
        \int_{\partial\Sigma_t}|\nabla b|\,\di\ell
        \right)\di t .
\]
By the coarea formula in $M$,
\[
        \int_T^R t^{-n}
        \left(
        \int_{\Sigma_t}\frac{|B|^2}{|\nabla b|}\,\di\sigma
        \right)\di t
        =
        \int_{\{T\leq b\leq R\}}b^{-n}|B|^2\,\di V.
\]
Similarly, since $\nabla b$ is tangent to $\partial M$ and
$|\nabla^{\partial M}b|=|\nabla b|$ on $\partial M$, the coarea formula on
$\partial M$ gives
\[
        \int_T^R t^{2-n}
        \left(
        \int_{\partial\Sigma_t}|\nabla b|\,\di\ell
        \right)\di t
        =
        \int_{\partial M\cap\{T\leq b\leq R\}}
        b^{2-n}|\nabla b|^2\,\di\sigma .
\]
Hence
\[
        \int_{[T,R]} t^{n-2}\,\di F(t)\geq
        \frac12\int_{\{T\leq b\leq R\}}b^{-n}|B|^2\,\di V  \\
        +
        2\int_{\partial M\cap\{T\leq b\leq R\}}
        b^{2-n}|\nabla b|^2\,\di\sigma.
\]

It remains to bound the left-hand side uniformly in $R$. By integration by parts for the Stieltjes integral,
\[
\begin{aligned}
        \int_{[T,R]} t^{n-2}\,\di F(t)
        &=
        R^{n-2}F(R)-T^{n-2}F(T)
        -(n-2)\int_T^R t^{n-3}F(t)\,\di t\\
        &=
        R^{n-2}F(R)-T^{n-2}F(T)
        -(n-2)(A(R)-A(T)).
\end{aligned}
\]

Because $F(R)\leq0$ and $A(R)\geq0$, it follows that
\[
        \int_{[T,R]} t^{n-2}\,\di F(t)
        \leq
        -T^{n-2}F(T)+(n-2)A(T).
\]
Thus
\[
        \int_{[T,\infty)} t^{n-2}\,\di F(t)<\infty.
\]
Consequently,
\[
        \int_{\{b\geq T\}}b^{-n}|B|^2\,\di V<\infty,
\quad
        \int_{\partial M\cap\{b\geq T\}}
        b^{2-n}|\nabla b|^2\,\di\sigma<\infty.
\]
We now derive a contradiction. Recall the identity
\begin{align}\label{conservative law}
       t^{1-n} \int_{\Sigma_t}|\nabla b|\,\di\sigma
        =c.
\end{align}
It follows from the coarea formula that
\[
        \int_{\{T\leq b\leq R\}}b^{1-n}|\nabla b|^2\,\di V
        =
        \int_T^R s^{1-n}\left(\int_{\Sigma_s}|\nabla b|\,\di\sigma\right)\di s 
        =
        c\int_T^R \di s
        =
        c(R-T).
\]
Hence
\begin{align}\label{linear-divergence}
        \int_{\{b\geq T\}}b^{1-n}|\nabla b|^2\,\di V=\infty.
\end{align}

We next construct a compactly supported test function which captures this
divergent quantity. First assume $n>3$ and set
\[
        \alpha:=\frac{3-n}{2}<0.
\]
Let $\eta_R=\eta_R(b)$ be a smooth cutoff satisfying
\[
        \eta_R=0\quad\text{on }\{b\leq T\},\qquad
        \eta_R=1\quad\text{on }\{2T\leq b\leq R\},
\]
\[
        \eta_R=0\quad\text{on }\{b\geq 2R\},
\]
and
\[
        |\eta_R'|\leq \frac{C}{T}\quad\text{on }[T,2T],
        \qquad
        |\eta_R'|\leq \frac{C}{R},\quad
        |\eta_R''|\leq \frac{C}{R^2}
        \quad\text{on }[R,2R].
\]
Define
\[
        \phi_R:=\eta_R(b)b^\alpha.
\]
On the region $\{2T\leq b\leq R\}$, we have $\phi_R=b^\alpha$, and hence
\[
        |\nabla\phi_R|^2
        =
        \alpha^2 b^{2\alpha-2}|\nabla b|^2
        =
        \alpha^2 b^{1-n}|\nabla b|^2.
\]
Therefore
\[
        \int_M|\nabla\phi_R|^2\,\di V
        \geq
        \alpha^2
        \int_{\{2T\leq b\leq R\}}
        b^{1-n}|\nabla b|^2\,\di V  
        \geq cR-C.
\]
In particular,
\[
        \int_M|\nabla\phi_R|^2\,\di V\to\infty
        \qquad\text{as } R\to\infty.
\]

On the other hand, applying the trace Poincar\'e inequality Theorem \ref{app:trace-poincare} to the scalar
function $|\nabla\phi_R|$ and using Kato's inequality gives
\[
        \int_M|\nabla\phi_R|^2\,\di V
        \leq
        \frac2n\int_{\partial M}|\nabla\phi_R|^2\,\di\sigma
        +
        \int_M|\Hess\phi_R|^2\,\di V .
\]

We first estimate the boundary term. Along $\partial M$, the Neumann condition
gives $\partial_\nu b=0$, and hence $\partial_\nu\phi_R=0$. Thus
$|\nabla\phi_R|=|\nabla^{\partial M}\phi_R|$ on $\partial M$. Away from the
cutoff regions,
\[
        |\nabla^{\partial M}\phi_R|^2
        \leq Cb^{1-n}|\nabla b|^2,
\]
and the cutoff errors are supported in
\[
        \{T\leq b\leq2T\}\cup\{R\leq b\leq2R\}.
\]
The inner cutoff contribution is bounded by a constant depending only on $T$.

On the outer cutoff region
\[
        \mathcal A_R^\partial:=\partial M\cap\{R\leq b\leq2R\},
\]
we have
\[
        \nabla^{\partial M}\phi_R
        =
        \left(\eta_R'(b)b^\alpha
        +
        \alpha\eta_R(b)b^{\alpha-1}\right)\nabla^{\partial M}b.
\]
Since $\nabla^{\partial M}b=\nabla b$ on $\partial M$ and
$|\eta_R'|\leq C/R$, it follows that
\[
\begin{aligned}
        |\nabla^{\partial M}\phi_R|^2
        &\leq
        C\left(|\eta_R'|^2b^{2\alpha}
        +
        b^{2\alpha-2}\right)|\nabla b|^2  \\
        &\leq
        Cb^{2\alpha-2}|\nabla b|^2
        =
        Cb^{1-n}|\nabla b|^2 .
\end{aligned}
\]
Since $b\geq T$,
\[
        b^{1-n}\leq T^{-1}b^{2-n}.
\]
Using the finite boundary energy estimate, we get
\[
\begin{aligned}
        \int_{\mathcal A_R^\partial}
        |\nabla^{\partial M}\phi_R|^2\,\di\sigma
        &\leq
        C\int_{\partial M\cap\{R\leq b\leq2R\}}
        b^{1-n}|\nabla b|^2\,\di\sigma  \\
        &\leq
        CT^{-1}
        \int_{\partial M\cap\{R\leq b\leq2R\}}
        b^{2-n}|\nabla b|^2\,\di\sigma  \\
        &\leq
        CT^{-1}
        \int_{\partial M\cap\{b\geq T\}}
        b^{2-n}|\nabla b|^2\,\di\sigma
        \leq C(T).
\end{aligned}
\]
Consequently,
\[
        \int_{\partial M}|\nabla\phi_R|^2\,\di\sigma\leq C(T)
\]
uniformly in $R$.

It remains to estimate the Hessian term. On the region where $\eta_R=1$,
\[
        \Hess(b^\alpha)
        =
        \alpha b^{\alpha-1}\Hess b
        +
        \alpha(\alpha-1)b^{\alpha-2}db\otimes db .
\]
Using
\[
        B=\Hess b^2-2|\nabla b|^2g
        =
        2b\Hess b+2db\otimes db-2|\nabla b|^2g,
\]
we solve for $\Hess b$:
\[
        \Hess b
        =
        \frac{1}{2b}B
        -
        \frac{1}{b}db\otimes db
        +
        \frac{|\nabla b|^2}{b}g .
\]
Thus
\[
        |\Hess(b^\alpha)|^2
        \leq
        Cb^{2\alpha-4}|B|^2
        +
        Cb^{2\alpha-4}|\nabla b|^4.
\]
Since $2\alpha-4=-n-1$, this becomes
\[
        |\Hess(b^\alpha)|^2
        \leq
        Cb^{-n-1}|B|^2
        +
        Cb^{-n-1}|\nabla b|^4.
\]
The first term is integrable because
\[
        b^{-n-1}\leq T^{-1}b^{-n}
        \quad\text{on } \{b\geq T\},
\]
and
\[
        \int_{\{b\geq T\}}b^{-n}|B|^2\,\di V<\infty.
\]
For the second term, the Cheng--Yau gradient estimate for the Green distance
function gives, after increasing $T$ if necessary,
\[
        |\nabla b|\leq C
        \quad\text{on } \{b\geq T\}.
\]
Hence
\[
        b^{-n-1}|\nabla b|^4
        \leq Cb^{-n-1}|\nabla b|^2.
\]
By the coarea formula and the flux identity,
\[
\begin{aligned}
        \int_{\{b\geq T\}}b^{-n-1}|\nabla b|^2\,\di V
        &=
        \int_T^\infty s^{-n-1}
        \left(
        \int_{\Sigma_s}|\nabla b|\,\di\sigma
        \right)\di s  \\
        &=
        C\int_T^\infty s^{-n-1}s^{n-1}\,\di s
        =
        C\int_T^\infty s^{-2}\,\di s
        <\infty.
\end{aligned}
\]

The Hessian contributions from the cutoff regions are also uniformly bounded.
The inner cutoff region $\{T\leq b\leq2T\}$ is compact, so its contribution is
bounded by a constant depending only on $T$. For the outer cutoff region
\[
        \mathcal A_R:=\{R\leq b\leq2R\},
\]
write
\[
        \phi_R=\psi_R(b),
        \qquad
        \psi_R(s):=\eta_R(s)s^\alpha.
\]
Then
\[
        \nabla^2\phi_R
        =
        \psi_R''(b)\,db\otimes db
        +
        \psi_R'(b)\,\nabla^2 b.
\]
On $\mathcal A_R$, we have $R\leq b\leq2R$, so
\[
        |\psi_R'(b)|\leq Cb^{\alpha-1},
        \qquad
        |\psi_R''(b)|\leq Cb^{\alpha-2}.
\]
Therefore
\[
        |\nabla^2\phi_R|^2
        \leq
        Cb^{2\alpha-4}|\nabla b|^4
        +
        Cb^{2\alpha-2}|\nabla^2b|^2.
\]
Using again the expression of $\nabla^2b$ in terms of $B$,
\[
        |\nabla^2b|^2
        \leq
        Cb^{-2}|B|^2
        +
        Cb^{-2}|\nabla b|^4,
\]
we obtain
\[
        |\nabla^2\phi_R|^2
        \leq
        Cb^{-n-1}|B|^2
        +
        Cb^{-n-1}|\nabla b|^4
        \quad\text{on } \mathcal A_R.
\]
Consequently,
\[
\begin{aligned}
        \int_{\mathcal A_R}|\nabla^2\phi_R|^2\,\di V
        &\leq
        C\int_{\mathcal A_R}b^{-n-1}|B|^2\,\di V
        +
        C\int_{\mathcal A_R}b^{-n-1}|\nabla b|^4\,\di V .
\end{aligned}
\]
The first term is bounded by the finite energy estimate. For the second term,
using $|\nabla b|\leq C$ and the flux identity,
\[
\begin{aligned}
        \int_{\mathcal A_R}b^{-n-1}|\nabla b|^4\,\di V
        &\leq
        C\int_R^{2R}s^{-n-1}
        \left(
        \int_{\Sigma_s}|\nabla b|\,\di\sigma
        \right)\di s  \\
        &=
        C\int_R^{2R}s^{-2}\,\di s
        \leq \frac{C}{R}.
\end{aligned}
\]
Thus
\[
        \int_{\mathcal A_R}|\nabla^2\phi_R|^2\,\di V\leq C(T)
\]
uniformly in $R$. Combining the estimates on the main region and the cutoff
regions, we conclude that
\[
        \int_M|\Hess\phi_R|^2\,\di V\leq C(T)
\]
uniformly in $R$.

Therefore the trace Poincar\'e inequality Theorem \ref{app:trace-poincare} gives
\[
        \int_M|\nabla\phi_R|^2\,\di V\leq C(T)
\]
uniformly in $R$, contradicting the lower bound
\[
        \int_M|\nabla\phi_R|^2\,\di V\geq cR-C.
\]
Hence the case
\[
        t^{3-n}A'(t)\leq0
\]
for all sufficiently large $t$ is impossible when $n>3$.

It remains to treat the minor modification when $n=3$. In this case
$\alpha=(3-n)/2=0$, so $b^\alpha$ is constant. Instead, we take
\[
        \phi_R:=\eta_R(b)\log b.
\]
On the main region $\{2T\leq b\leq R\}$,
\[
        |\nabla\phi_R|^2
        =
        b^{-2}|\nabla b|^2
        =
        b^{1-n}|\nabla b|^2,
\]
and the same flux computation gives
\[
        \int_M|\nabla\phi_R|^2\,\di V\geq cR-C.
\]
Moreover,
\[
        \Hess(\log b)
        =
        b^{-1}\Hess b-b^{-2}db\otimes db.
\]
Using the same expression for $\Hess b$ in terms of $B$, we obtain
\[
        |\Hess(\log b)|^2
        \leq
        Cb^{-4}|B|^2
        +
        Cb^{-4}|\nabla b|^4.
\]
When $n=3$, the finite scale-invariant energy is
\[
        \int_{\{b\geq T\}}b^{-3}|B|^2\,\di V<\infty.
\]
Since $b^{-4}\leq T^{-1}b^{-3}$, the first term is integrable. The second
term is controlled by the Cheng--Yau estimate and the flux identity:
\[
\begin{aligned}
        \int_{\{b\geq T\}}b^{-4}|\nabla b|^4\,\di V
        &\leq
        C\int_{\{b\geq T\}}b^{-4}|\nabla b|^2\,\di V  \\
        &=
        C\int_T^\infty s^{-4}
        \left(
        \int_{\Sigma_s}|\nabla b|\,\di\sigma
        \right)\di s  \\
        &=
        C\int_T^\infty s^{-4}s^2\,\di s
        <\infty .
\end{aligned}
\]
The boundary and cutoff terms are controlled exactly as above. Hence the trace
Poincar\'e inequality Theorem \ref{app:trace-poincare} again gives a uniform upper bound for
\[
        \int_M|\nabla\phi_R|^2\,\di V,
\]
contradicting its linear divergence. This proves the contradiction in all
dimensions $n\geq3$.
\end{proof}

\appendix
\section{Trace Poincar\'e Inequality}\label{app:trace-poincare}

We record the trace Poincar\'e inequality used in the proof of our main results.

\begin{theorem}\label{Trace Poincare inequality}
Let $(M^n,g)$ be a complete Riemannian $n$-manifold with boundary
$\del M$, $H=\tr_g h$ the mean curvature of $\del M$. Assume that
\[
        \Ric_g\ge 0, \quad H\ge n-1.
\]
Then for every
$W^{1,2}$ function $f$ on $M$, it holds
\[
        \int_M f^2\,\di V
        \le
        \frac{2}{n}\int_{\del M} f^2\,\di\sigma
        +
        \frac{n+1}{2n^2}\int_M |\nabla f|^2\,\di V .
\]
Here $f|_{\del M}$ is understood in the trace sense.
\end{theorem}
 Our proof is inspired by Sakurai \cite{Sakurai}*{Lemma 7.3}. But we use a different language, the needle decomposition, to write the proof. We refer to Klartag \cite{KlartagNeedle}, Cavalletti--Mondino \cite{CavallettiMondino} for the basics on the needle decomposition, also known as the 1D localization technique.

\begin{proof}
 Since
$\Ric_g\ge 0$ and $H\ge n-1$, the distance function to the boundary
\[
    \dist_{\del M}(x):=\operatorname{dist}(x,\del M)
\]
induces a decomposition of $(M,g,\vol)$ into transport rays
up to a negligible set. More precisely, there exists an index space $\mathfrak Q$, identified as $\del M$, a Borel measure $\mathfrak q$ on $\mathfrak Q$, and a family of unit speed minimizing geodesics
\[
        \gamma_\alpha:[0,l_\alpha]\to M,
        \qquad \alpha\in \mathfrak Q,
\]
such that
\[
        \gamma_\alpha(0)\in \del M,
        \qquad
        \dist_{\del M}(\gamma_\alpha(t))=t,
\]
Moreover, the inscribed radius estimate of Li \cite{LiMartin2014}*{Theorem 1.1} gives
\[
        l_\alpha\le 1
\]
for $\mathfrak q$-a.e. $\alpha$.
There is a disintegration of $\vol$ associated to this decomposition as 
\[
        \int_M \varphi\,\di V
        =
        \int_Q\int_0^{l_\alpha}
        \varphi(\gamma_\alpha(t))h_\alpha(t)\,\di t\,\di \mathfrak q(\alpha).
\] 
In the meantime, the surface disintegrates as
\[
        \int_{\del M}\psi\,\di\sigma
        =
        \int_{\mathfrak Q}
        \psi(\gamma_\alpha(0))h_\alpha(0)\,\di \mathfrak q(\alpha).
\]
See Ketterer \cite{KettererHK}*{Remark 5.4}. The density $h_\alpha$ is a $\mathrm{CD}(0,n)$ density along the ray $\gamma_\alpha$, in particular $h_\alpha^{\frac{1}{n-1}}$ is concave. The curvature assumption $H\ge (n-1)$ gives the boundary derivative estimate
\[
 \frac{\di}{\di t}\bigg|_{t=0+}\log h_\alpha(t) =-H (\gamma_\alpha(0))\le -(n-1).
\]
Let $u_\alpha(t):=h_\alpha(t)^{\frac{1}{n-1}}$,
then $u$ is concave with
\[
 \frac{\di}{\di t}\bigg|_{t=0+} u_\alpha(t)\le -u_\alpha(0).
\]
The boundary estimates gives a interior one through concavity.

\begin{lemma}
For $\mathfrak q$-a.e. $\alpha$ and every $0\le t\le l_\alpha$, it holds
\[
        h_\alpha(t)
        \le
        (1-t)^{n-1}h_\alpha(0).
\]
\end{lemma}

\begin{proof}
Fix a ray and omit the subscript $\alpha$. Recall that $u=u_\alpha$ is concave and satisfies
\[
        u'(0+)\le -u(0).
\]
By concavity,
\[
        u(t)-u(0)\le u'(0+)t\le -u(0)t.
\]
Hence
\[
        u(t)\le (1-t)u(0).
\]
Substituting back gives
\[
        h(t)\le (1-t)^{n-1}h(0).
\]
\end{proof}
We now prove the one dimensional trace Poincar\'e inequality along each ray. For a smooth function $f$, fundamental theorem of calculus and Cauchy--Schwarz gives 
\begin{align*}
    f^2(\gamma_\alpha(t))\le 2f^2(\gamma_\alpha(0))
        +
    2\left(\int_0^t |\nabla f|(\gamma_\alpha(s))\,\di s\right)^2\le 2f^2(\gamma_\alpha(0))
        +2t\int_0^t |\nabla f|^2(\gamma_\alpha(s))\,\di s
\end{align*}
        
Multiplying by $h_\alpha(t)$ and integrating in $t$, we obtain
\[
        \int_0^{l_\alpha}
        f^2(\gamma_\alpha(t))h_\alpha(t)\,\di t
        \le
        2f^2(\gamma_\alpha(0))
        \int_0^{l_\alpha}h_\alpha(t)\,\di t
        +
        2\int_0^{l_\alpha} t h_\alpha(t)\int_0^t |\nabla f|^2(\gamma_\alpha(s))\,\di s\,\di t .
\]
We estimate the two terms separately. By the previous lemma and
$l_\alpha\le 1$,
\[
        2f^2(\gamma_\alpha(0))
        \int_0^{l_\alpha}h_\alpha(t)\,\di t
        \le
        2f^2(\gamma_\alpha(0))h_\alpha(0)
        \int_0^1(1-t)^{n-1}\,\di t       
        =
        \frac{2}{n}
        f^2(\gamma_\alpha(0))h_\alpha(0).
\]
For the second term, we switch integrals by Fubini.
\[
        2\int_0^{l_\alpha}
        t h_\alpha(t)
        \int_0^t |\nabla f|^2(\gamma_\alpha(s))\,\di s\,\di t        \\
        =
        2\int_0^{l_\alpha}
        \left(
        \int_s^{l_\alpha} t h_\alpha(t)\,\di t
        \right)
        |\nabla f|^2(\gamma_\alpha(s))\,\di s .
\]
Since $h_\alpha$ satisfies the $\mathrm{CD}(0,n)$ density estimate $h_\alpha(t)(1-t)^{1-n}\le h_\alpha(s)(1-s)^{1-n}$ when $l_\alpha\le 1$, we have
\[
\begin{aligned}
    \int_s^{l_\alpha} t h_\alpha(t)\,\di t
       & \le
        h_\alpha(s)\int_s^{l_\alpha}(1-s)^{1-n}(1-t)^{n-1} t\,\di t\\
        &\le
        h_\alpha(s)\int_s^1 (1-s)^{1-n}(1-t)^{n-1} t\,\di t\\
        &=\left(\frac{1-s}{n}-\frac{(1-s)^2}{n+1}\right) h_\alpha(s)\le\frac{n+1}{4n^2}h_\alpha(s).
\end{aligned}     
\]
Hence
\[
        2\int_0^{l_\alpha}
        \left(
        \int_s^{l_\alpha} t h_\alpha(t)\,\di t
        \right)
        |\nabla f|^2(\gamma_\alpha(s))\,\di s 
        \le
        \frac{n+1}{2n^2}\int_0^{l_\alpha}
        |\nabla f|^2(\gamma_\alpha(s))h_\alpha(s)\,\di s .
\]
Combining the two estimates, we obtain for $q$-a.e. $\alpha$,
\[
        \int_0^{l_\alpha}
        f^2(\gamma_\alpha(t))h_\alpha(t)\,\di t
        \le
        \frac{2}{n}
        f^2(\gamma_\alpha(0))h_\alpha(0)
        +
        \frac{n+1}{2n^2}\int_0^{l_\alpha}
        |\nabla f|^2(\gamma_\alpha(t))h_\alpha(t)\,\di t .
\]
Finally, integrate over $Q$ and use the disintegration formulas:
\[
\begin{aligned}
        \int_M f^2\,\di V
        &=
        \int_Q\int_0^{l_\alpha}
        f^2(\gamma_\alpha(t))h_\alpha(t)\,\di t\,dq(\alpha)        \\
        &\le
        \frac{2}{n}
        \int_Q f^2(\gamma_\alpha(0))h_\alpha(0)\,dq(\alpha)
        +
        \frac{n+1}{2n^2}\int_Q\int_0^{l_\alpha}|\nabla f|^2(\gamma_\alpha(t))h_\alpha(t)\,\di t\,dq(\alpha)        \\
        &=
        \frac{2}{n}
        \int_{\del M} f^2\,\di\sigma
        +
        \frac{n+1}{2n^2}\int_M |\nabla f|^2\,\di V .
\end{aligned}
\]
This proves the inequality for smooth functions. The general
$W^{1,2}$ case follows by the usual approximation argument and the
continuity of the trace map.
\end{proof}

The similar proof also derives the $L^p$ version of trace Poincar\'e for $p\in[1,\infty]$. We mention that if we do not use $H\ge n-1$ but keep $H\ge 0$ as a variable in the computation, we can get a sharp $L^1$ trace Poincar\'e inequality. Let $\pi: M\to \mathfrak Q\cong \partial M$ be the projection map that takes $\gamma_\alpha(t)$, $t\in[0,l_\alpha]$, to $\gamma_\alpha(0)\in \del M$, then the $L^1$ Poincar\'e takes the following form.

\begin{proposition}
    Let $(M^n,g)$ be a complete Riemannian $n$-manifold with boundary
$\del M$. Assume that
\[
        \Ric_g\ge 0,  \qquad  H\ge 0.
\]
Then for every
$W^{1,1}$ function $f$ on $M$, it holds
\[
        \int_M |f|\,\di V
        \le
        \frac{n-1}{n}\int_{\del M} \frac{|f|}{H}\,\di\sigma
        +
        \frac{1}{n}\int_M \left(\frac{n-1}{H(\pi(x))}-{\dist_{\del M}(x)}\right)|\nabla f|\,\di V .
\]
\end{proposition}
Here, the geodesic-wise inscribed radius estimate \cite{LiMartin2014}*{Proof of Theorem 1.1, (3.2)} implies 
\[
\frac{n-1}{H(\pi(x))}-{\dist_{\del M}(x)}\ge 0.
\]
We skip the similar but more tedious proof. Instead, we point out that when $M$ is compact, taking $f=1$ recovers the Ros inequality \cite{Wang2021}*{Theorem 6}: 
\[
\vol(M)\le \int_{\del M}\frac{n-1}{nH}\, \di \sigma.
\]

\bibliographystyle{amsalpha} 
\bibliography{new}

\end{document}